\documentclass[12pt]{amsart}

\usepackage{amsmath, amssymb}
\usepackage{mathtools}
\usepackage{bbm}
\usepackage[utf8]{inputenc}
\usepackage{calrsfs}
\DeclareMathAlphabet{\pazocal}{OMS}{zplm}{m}{n}
\usepackage{color}
\usepackage{overpic}
\usepackage{hyperref}
\usepackage{soul}
\usepackage{dsfont}
\usepackage{enumitem}

\usepackage[english]{babel}

\usepackage{array}

\usepackage{caption}
\makeatletter
\@namedef{subjclassname@2020}{\textup{2020} Mathematics Subject Classification}
\makeatother

\usepackage{float}

\newtheorem{theorem}{Theorem}[section]

\newtheorem{cor}[theorem]{Corollary}

\theoremstyle{definition}

\theoremstyle{remark}

\numberwithin{equation}{section}

\allowdisplaybreaks

\begin{document}

\title[Discrete Hyperbolic Secant Distributions]{Discrete Hyperbolic Secant Distributions}



\author{Leonard J. E. Pleschberger}
\address{Leonard Jobst Eberhard Pleschberger, Neptunstra\ss e 18, D\"usseldorf D-40223 \& Sechshauser Stra\ss e 24/1/15, A-1150 Vienna.}
\email{leonard.pleschberger@gmail.com}

\date{\today}

\maketitle

\begin{abstract}
We introduce a family of discrete hyperbolic secant distributions on $\mathbb{Z}$, whose normalizing constants arise from series values calculated by Ramanujan and are expressed in terms of Gau\ss' constant $G=\varpi/\pi$, where $\varpi=\Gamma^2(1/4)/(2\sqrt{2\pi})$ is the lemniscate constant. Using the elliptic lambda-star function $\lambda^*$, we construct scaled versions of these distributions parametrized by $\sqrt{r}$ and $1/\sqrt{r}$ for $r \in \mathbb{N}$. For the first such distribution, we compute the moments up to the eighth degree in closed form via Poisson summation, exploiting that the hyperbolic secant is a fixed point of the $-2\pi i$-Fourier transform. As a byproduct, we obtain closed-form values for series of odd powers of the hyperbolic secant up to the ninth degree, e.\,g.\ $\sum_{k \in \mathbb{Z}} \operatorname{sech}^3(\pi k) = (G^3+G)/\sqrt{2}$, and reinterpret several classical Ramanujan series probabilistically. Finally, we apply our results to evaluate a contour integral, on the critical line, of the product of Dirichlet's beta, gamma, and Riemann zeta functions.
\end{abstract}

\vspace{-0.5cm}

\section{Introduction}

In Section 2, we introduce a family of discrete hyperbolic secant distributions. The normalizing constants are derived from values of series calculated by Ramanujan. Using Gau\ss' constant $G=\varpi/\pi$ with the lemniscate constant $\varpi=\Gamma^2(1/4)/(2\sqrt{2 \pi})$, we can represent these distributions in very basic terms, e.\,g.
$$
\left(\mathbb{P} (X = k)\right)_{k \in \mathbb{Z}} = \left(\frac{\operatorname{sech}\left(\pi k\right)}{\sqrt{2}G} \right)_{k\in\mathbb{Z}}
$$
and versions with a scaling factor of $\sqrt{r}$ and $1/\sqrt{r}$ applied to $\pi k$. We use Ramanujan series and apply the lambda-star function $\lambda^*$ which yields the required singular values of the occurring elliptic integrals.\\

In Section 3, we calculate the moments up to the eighth degree for the first hyperbolic secant distribution, e.\,g. $\mathbb{E}\left[X^2 \right] = \operatorname{Var}(X) = (G/2)^2$. Therefrom, we derive the values of series of odd multiplicities of the hyperbolic secant up to the ninth degree, e.\,g.
$$
\sum_{k\in\mathbb{Z}} \mathrm{sech}^3 (\pi k) = \left(G^3 + G\right)/\sqrt{2}.
$$
In this case, the hyperbolic secant is a fix-point of the $-2\pi i$-Fourier transform which allows us comfortable calculations via Poisson summation.\\

In Section 4, we give further statistical interpretations of some Ramanujan series in terms of the second hyperbolic secant distribution, e.\,g. $\mathbb{E}\left[(2Y+1)^2 \right] = 2G^2$.\\

In Section 5, we apply the first hyperbolic secant distribution to contour integrals. We calculate a contour integral for a product $\Xi$ of Dirichlet's beta function, the gamma function and Riemann's zeta function on the critical line, viz.
$$
\int_{1/2 - i\cdot \infty}^{1/2 + i\cdot \infty} \pi^{-(s+2)} \beta(s+2) \Gamma(s+2) \zeta(s) \, \mathrm{d}s = i \left( \frac{\varpi}{\sqrt{2}} \mathbb{E}\left[X^2 \right] - \frac{\pi}{8} \right) = i \left( \frac{\varpi}{\sqrt{2}} \left(\frac{G}{2}\right)^2 - \frac{\pi}{8} \right).
$$
We use the Mellin transform and the residue theorem.

\section{Basic Discrete Hyperbolic Secant Distributions}

We introduce a family of six discrete hyperbolic secant distributions through their probability mass functions. These functions are normalized by series values given in \cite{ramanujan_1991} by using the lemniscate modulus $k_1 = 1/\sqrt{2}$. We represent the normalizing constants in terms of the Gau\ss ian constant $G:=\varpi/\pi$ with the lemniscate constant $\varpi:=\Gamma^2(1/4)/(2\sqrt{2 \pi})$. This results in particularly simple expressions.

\begin{theorem} \label{thm:pmfs}
We tabulate six discrete hyperbolic secant probability distributions:

\renewcommand{\arraystretch}{2}
\begin{table}[h]
    \centering
    \begin{tabular}{c|c|c|c}
PMF  &   Variable & Constant  &  Law \\
\hline
$(p_n)_{n \in \mathbb{N}_0}$ & $X_u$ & $C_{1,u}$ & $\frac{2}{\sqrt{2}G  + 1} \operatorname{sech} \left(\pi n\right)$\\
$(q_k)_{k \in \mathbb{Z}}$ & $X_b$ & $C_{1,b}$ &  $\frac{1}{\sqrt{2}G} \operatorname{sech}\left(\pi k\right)$\\
$(r_n)_{n \in \mathbb{N}_0}$ &  $Y_u$ & $C_{2,u}$ &  $\frac{2}{G} \operatorname{sech} \left(\frac{\pi}{2} (2n+1)\right)$\\
$(s_k)_{k \in \mathbb{Z}}$ &  $Y_b$ & $C_{2,b}$ &  $\frac{1}{G} \operatorname{sech}\left(\frac{\pi}{2} (2k+1)\right)$\\
$(t_n)_{n \in \mathbb{N}_0}$ & $Z_u$ & $C_{3,u}$ &  $\frac{2}{(1+\sqrt{2})G + 1} \operatorname{sech}\left(\frac{\pi}{2} n\right)$\\
$(u_k)_{k \in \mathbb{Z}}$ &  $Z_b$ & $C_{3,b}$ & $\frac{1}{(1+\sqrt{2})G}  \operatorname{sech}\left(\frac{\pi}{2} k\right)$
\end{tabular}
\end{table}

In full length, we name them the first, second and third discrete uni- or bilateral hyperbolic secant distributions and omit adjectives if the context is clear.
\end{theorem}

\begin{proof}
The normalizing constants can be deduced by the astonishing work of Ramanujan: He wrote in his \textit{Entry 17 (i)} in Chapter $17$ of his second notebook [in modern notation] his identity
\begin{equation}\label{form:ramanujan_1}
1 + 2 \sum_{n=1}^\infty \operatorname{sech}\left(\frac{K'(k)}{K(k)} \pi n\right) = \frac{2}{\pi} K(k), \quad \mathrm{with} \quad K(k) = \int_0^{\pi/2} \frac{\mathrm{d}\vartheta}{\sqrt{1-k^2\sin^2(\vartheta)}},
\end{equation}
Jacobi's complete elliptic integral of the first kind. We use the usual notation $K'(k) = K(k') = K(\sqrt{1-k^2})$ with the elliptic modulus $k$. The identity is proven by Berndt, cf. p. 138 of \cite{ramanujan_1991}. The $\lambda^\star$-function gives the value for the elliptic modulus $\lambda^\star(r)=k_r$ for which
\begin{equation}\label{form:elliptic_modulus}
\frac{K'(\lambda^\star(r))}{K(\lambda^\star(r))} = \sqrt{r} \quad \mathrm{with} \quad k_1 = \lambda^\star(1) = \frac{1}{\sqrt{2}}, \quad K(k_1) = \frac{\varpi}{\sqrt{2}}
\end{equation}
holds. For all the analysis, cf. the classical text  \cite{whittaker_1920}. Hence, by taking $r=1$, we get
$$
\sum_{n=0}^\infty \operatorname{sech}(\pi n) = \frac{\varpi}{\sqrt{2} \cdot \pi} + \frac{1}{2} = \frac{\sqrt{2} G + 1}{2} \ \Rightarrow \ C_{1,u} = \frac{2}{\sqrt{2}G + 1} \ \Rightarrow \ C_{1,b} = \frac{1}{\sqrt{2}G}
$$
due to the symmetry. In \textit{Entry 16 (ix)} of Chapter 17 of his second notebook, Ramanujan stated [in modern notation] the identity
$$
\sum_{n=0}^\infty \operatorname{sech}\left(\frac{\pi}{2} (2n+1)\right) = \frac{\varpi}{2 \pi} = \frac{G}{2}, \quad \Rightarrow \quad C_{2,u} = \frac{2}{G}.
$$
This is proven by Berndt on p. 134 of \cite{ramanujan_1991} and by Campbell in Theorem 7 of \cite{campbell_2023}. Fom  the absolute convergence follows
\begin{align*}
 & \sum_{n=0}^\infty \operatorname{sech}\left(\frac{\pi}{2}\cdot 2n \right) + \sum_{n=0}^\infty \operatorname{sech}\left(\frac{\pi}{2} \cdot (2n + 1)\right) = \frac{(1+\sqrt{2})G + 1}{2}.
\end{align*}
Hence, we get
$$
C_{3,u} = \frac{2}{(1+\sqrt{2})G + 1} \quad \mathrm{and} \quad C_{3,b} = \frac{1}{(1+\sqrt{2})G},
$$
where the latter was proven on p. 137 in \cite{ramanujan_1991}. Finally, 
$$
\sum_{k \in \mathbb{Z}} \operatorname{sech}\left(\frac{\pi}{2}(2k+1)\right) = \sum_{k \in \mathbb{Z}} \operatorname{sech}\left(\frac{\pi}{2} \cdot k\right) - \sum_{k\in \mathbb{Z}} \operatorname{sech} \left(\frac{\pi}{2} \cdot 2k\right) = G, \ \Rightarrow \ C_{2,b} = \frac{1}{G}
$$
finishes the proof.
\end{proof}

\vspace{1cm}

If one wanted to exaggerate, one could interpret $1 + \sqrt{2}$  as the silver ratio, but we tend to regard the expression merely as the sum of $1$ and $\sqrt{2}$.\\

The proof of the last theorem just handled the lemniscate case for $r=1$ in (\ref{form:elliptic_modulus}). In the following, we use the notations introduced in this proof. Now, let $r \ge 1$. By Ramanujan's formula \ref{form:ramanujan_1}, we immediately get a scaled family of discrete hyperbolic secant distributions.

\begin{cor}
Let $r \in \mathbb{N}$. Then, the formula
$$
\left(p_{r,k}\right)_{k\in\mathbb{Z}} = \mathbb{P}(X=k) \left(\frac{\pi}{2K(k_r)}\operatorname{sech} (\sqrt{r}\pi k)\right)_{k\in\mathbb{Z}}
$$
is the probability mass function of a random variable $X \sim \mathsf{Sech}(\sqrt{r})$. We call this distribution the discrete bilateral hyperbolic secant distribution with scaling factor $\sqrt{r}$.
\end{cor}

We tabular some interesting singular values $k_r$ with corresponding probability mass functions $(p_{r,k})_{k \in \mathbb{Z}}$

\renewcommand{\arraystretch}{2}
\begin{table}[h]
    \centering
    \begin{tabular}{c|c|r}
$p_{r,k}$  &   $ \lambda^*(r) = k_r$ & $\frac{\pi}{2K(k_r)}\operatorname{sech} (\sqrt{r}\pi k)$\\
\hline
$\mathbf{p_{1,k}}$  & $\frac{1}{\sqrt{2}}$ & $\frac{1}{\sqrt{2}G} \operatorname{sech} (\pi k)\quad \ $\\
$p_{2,k}$  & $\sqrt{2}-1$ & $\frac{4 \pi\sqrt{2\pi}}{\sqrt{\sqrt{2}+1} \cdot \Gamma(1/8)\Gamma(3/8)} \operatorname{sech} (\sqrt{2}\pi k)$\\
$p_{3,k}$  & $\frac{\sqrt{2}\left(\sqrt{3}-1\right)}{4}$ & $\frac{\sqrt[4]{2^3} \cdot \pi^2 }{\sqrt[4]{3} \cdot \Gamma^3(1/3)}\operatorname{sech} (\sqrt{3}\pi k)$\\
$\mathbf{p_{4,k}}$  & $3-2\sqrt{2}$ & $\frac{2}{(1 + \sqrt{2})G}\operatorname{sech} (2\pi k)\quad$\\
$p_{5,k}$  & $\frac{1}{\sqrt{2}} - \sqrt{\sqrt{5}-2}$ & $\frac{2\pi \sqrt{10 \pi}}{\sqrt[4]{2+\sqrt{5}} \sqrt{\Gamma(1/20)\Gamma(3/20)\Gamma(7/20)\Gamma(9/20)}} \operatorname{sech} (\sqrt{5}\pi k)$\\
$p_{6,k}$  & $\left(2-\sqrt{3}\right)\left(\sqrt{3}-\sqrt{2}\right)$ & $\frac{4\pi \sqrt{6 \pi}}{\sqrt{(1 +\sqrt{2}+\sqrt{6}) \cdot \Gamma(1/24)\Gamma(5/24)\Gamma(7/24)\Gamma(11/24)}} \operatorname{sech} (\sqrt{6}\pi k)$\\
$p_{7,k}$  & $\frac{3-\sqrt{7}}{8}$ & $\frac{2\sqrt[4]{7} \cdot \pi^2}{\Gamma(1/7)\Gamma(2/7)\Gamma(4/7)} \operatorname{sech} (\sqrt{7}\pi k)$\\
$\mathbf{p_{9,k}}$  & $\frac{\left(\sqrt{2}-\sqrt[4]{3}\right)\left(\sqrt{3}-1\right)}{2}$ & $\frac{3}{\left(\sqrt[4]{3^3} + \sqrt[4]{3} \right) G} \operatorname{sech} (3\pi k)\quad$\\
$\mathbf{p_{16,k}}$  & $\left(1 + \sqrt{2}\right)^2 \left(\sqrt[4]{2} - 1\right)^4$ & $\frac{4}{\left(1 + \sqrt[4]{2} \right)^2 G} \operatorname{sech} (4\pi k)\quad$\\
$\mathbf{p_{25,k}}$  & $\frac{\left(3 - 2 \sqrt[4]{5}\right)\left(\sqrt{5} - 2\right)}{\sqrt{2}}$ & $\frac{5}{\left(2\sqrt{2} + \sqrt{10} \right) G} \operatorname{sech} (5\pi k) \quad$
\end{tabular}
    \caption{Note, that the probability mass functions for the square numbers in \textbf{bold} have quiet simple normalization constants in terms of Gau\ss' constant $G=\varpi/\pi$.}
    \label{tab:sqrt_r_sech}
\end{table}

We can also determine the probability mass functions for the DSH distribution with scaling factor $1/\sqrt{r}$. 
The DSH distribution makes an easy proof of the following functional equation possible, which was shown in \cite{kirschenhofer_1996} by using the Mellin transform. For $x > 0$, the functional equation
\begin{equation}\label{eq:functional_F}
F(x) := \sum_{k=1}^\infty \frac{e^{-kx}}{1 + e^{-2kx}} =  \frac{\pi}{4x} - \frac{1}{4} + \frac{\pi}{x}F\left(\frac{\pi^2}{x}\right)
\end{equation}
holds. In terms of the hyperbolic secant, we can also write 
$$
F(x) = \frac{1}{4} \sum_{k\in\mathbb{Z}} \operatorname{sech}(xk) - \frac{1}{4}
$$
and get the following scaling theorem:

\begin{theorem}
Let $x \in \mathbb{R}$. Then one has the functional equation
\begin{equation}
\sum_{k \in \mathbb{Z}} \operatorname{sech}\left(xk\right) = \frac{\pi}{x} \sum_{k \in \mathbb{Z}} \operatorname{sech}\left(\frac{\pi^2}{x} k \right)
\end{equation}
which is equivalent to (\ref{eq:functional_F}). For every $r \in \mathbb{N}$, this yields the probability mass function
$$
\left(q_{r,k} \right)_{k \in \mathbb{Z}} = \left(\frac{\pi}{\sqrt{r} \cdot 2K(k_r)} \operatorname{sech}\left( \frac{\pi}{\sqrt{r}} k \right) \right)_{k \in \mathbb{Z}}.
$$
for the discrete bilateral hyperbolic secant distribution with scaling factor $1/\sqrt{r}$.
\end{theorem}

\begin{proof}
The $-2i \pi$-Fourier transform $\mathcal{F}$ of the hyperbolic secant is well-known, viz.
$$
\mathcal{F}[u \mapsto \operatorname{sech}(u)](t) = \pi \operatorname{sech} \left(\frac{\pi}{2}t\right)
$$
Hence, by the Poisson summation formula, for every $x \in \mathbb{R}$ we get
\begin{align*}
\sum_{k \in \mathbb{Z}} \operatorname{sech}(xk) &=\ \sum_{k \in \mathbb{Z}} \mathcal{F}\left[y \mapsto \operatorname{sech}(xy)\right](\eta)\big|_{\eta = k}\\
&=\ \sum_{k\in \mathbb{Z}} \int_\mathbb{R} \operatorname{sech}(u) e^{-i(2\pi/x) \cdot \eta} \frac{\mathrm{d}u}{x}\big|_{\eta = k}\\
&=\ \frac{\pi}{|x|} \sum_{k \in \mathbb{Z}} \operatorname{sech}\left(\frac{\pi^2}{x}\eta\right) \big|_{\eta = k},
\end{align*}
which results in Equation (\ref{eq:functional_F}). Now, we substitute $x = \sqrt{r} \cdot \pi$. For every $r \in \mathbb{N}$, Ramanujan's formula \ref{form:ramanujan_1} yields
$$
\sum_{k\in \mathbb{Z}} \operatorname{sech}\left(\frac{\pi}{\sqrt{r}} k \right) = \sqrt{r} \left(\sum_{k \in \mathbb{Z}} \operatorname{sech}\left(\sqrt{r} \cdot \pi k \right)\right) = \frac{2\sqrt{r}}{\pi}K\left(k_r\right),
$$
as claimed.
\end{proof}

For the scaling factor $1/\sqrt{r}$, some relevant normalization constants can easily be calculated by multiplying the corresponding normalizing constants for the scaling factor $\sqrt{r}$ in Table \ref{tab:sqrt_r_sech} with the number $1/\sqrt{r}$. This yields some singular values and matching probability mass functions $\left(q_{r,k}\right)_{k \in \mathbb{Z}}$ as follows:

\renewcommand{\arraystretch}{2}
\begin{table}[t]
    \centering
    \begin{tabular}{c|c|r}
$q_{r,k}$  &   $ \lambda^*(r) = k_r$ & $\frac{\pi}{\sqrt{r} \cdot 2K(k_r)}\operatorname{sech} (\sqrt{r}\pi k)$\\
\hline
$\mathbf{q_{1,k}}$  & $\frac{1}{\sqrt{2}}$ & $\frac{1}{\sqrt{2}G} \operatorname{sech} \left(\pi k\right)\quad  $\\
$q_{2,k}$  & $\sqrt{2}-1$ & $\frac{4\sqrt{\pi^3}}{\sqrt{\sqrt{2}+1} \cdot \Gamma(1/8)\Gamma(3/8)} \operatorname{sech} \left(\frac{\pi}{\sqrt{2}} k\right)$\\
$q_{3,k}$  & $\frac{\sqrt{2}\left(\sqrt{3}-1\right)}{4}$ & $\frac{\sqrt[4]{2^3} \cdot \pi^2 }{\sqrt[4]{3^3} \cdot \Gamma^3(1/3)}\operatorname{sech} \left(\frac{\pi}{\sqrt{3}} k\right)$\\
$\mathbf{q_{4,k}}$  & $3-2\sqrt{2}$ & $\frac{1}{(1 + \sqrt{2})G}\operatorname{sech} \left(\frac{\pi}{2} k\right)\quad$\\
$q_{5,k}$  & $\frac{1}{\sqrt{2}} - \sqrt{\sqrt{5}-2}$ & $\frac{\sqrt{8 \pi^3}}{\sqrt[4]{2+\sqrt{5}} \sqrt{\Gamma(1/20)\Gamma(3/20)\Gamma(7/20)\Gamma(9/20)}} \operatorname{sech} \left(\frac{\pi}{\sqrt{5}} k\right)$\\
$q_{6,k}$  & $\left(2-\sqrt{3}\right)\left(\sqrt{3}-\sqrt{2}\right)$ & $\frac{\sqrt{16 \pi^3}}{\sqrt{(1 +\sqrt{2}+\sqrt{6}) \cdot \Gamma(1/24)\Gamma(5/24)\Gamma(7/24)\Gamma(11/24)}} \operatorname{sech} \left(\frac{\pi}{\sqrt{6}} k\right)$\\
$q_{7,k}$  & $\frac{3-\sqrt{7}}{8}$ & $\frac{2\pi^2}{\sqrt[4]{7} \cdot \Gamma(1/7)\Gamma(2/7)\Gamma(4/7)} \operatorname{sech} \left(\frac{\pi}{\sqrt{7}} k\right)$\\
$\mathbf{q_{9,k}}$  & $\frac{\left(\sqrt{2}-\sqrt[4]{3}\right)\left(\sqrt{3}-1\right)}{2}$ & $\frac{1}{\left(\sqrt[4]{3^3} + \sqrt[4]{3} \right) G} \operatorname{sech} \left(\frac{\pi}{3} k\right)\quad$\\
$\mathbf{q_{16,k}}$  & $\left(1 + \sqrt{2}\right)^2 \left(\sqrt[4]{2} - 1\right)^4$ & $\frac{1}{\left(1 + \sqrt[4]{2} \right)^2 G} \operatorname{sech} \left(\frac{\pi}{4} k\right)\quad$\\
$\mathbf{q_{25,k}}$  & $\frac{\left(3 - 2 \sqrt[4]{5}\right)\left(\sqrt{5} - 2\right)}{\sqrt{2}}$ & $\frac{1}{\left(2\sqrt{2} + \sqrt{10} \right) G} \operatorname{sech} \left(\frac{\pi}{5} k\right) \quad$
\end{tabular}
    \caption{Note, that the probability mass functions highlighted in \textbf{bold} have simple normalization constants with a $1$ in the nominator and scalar times  Gau\ss' constant $G=\varpi/\pi$.}
    \label{tab:sqrt_r_sech}
\end{table}

\vspace{3cm}

\section{The First Discrete Bilateral Hyperbolic Secant Distribution}

We calculate some moments for the first discrete bilateral hyperbolic secant distribution.

\begin{theorem} \label{thm_Var_1_b}
Let $X\sim\mathsf{Sech}_{1,b}$ be a discrete random variable which is distributed according to the first discrete bilateral hyperbolic secant distribution. The odd moments vanish and we have
\begin{align*}    
(i) \quad \mathbb{E}\left[X^2 \right] =&\ \mathrm{Var}(X) = \left(\frac{G}{2}\right)^2 \approx 0.1741505,\\
(ii) \quad \mathbb{E}\left[X^4 \right] =&\ 9\left(\frac{G}{2}\right)^4 \approx 0.2729555,\\
(iii) \quad \mathbb{E}\left[X^6 \right] =&\ 153\left(\frac{G}{2}\right)^6 \approx 0.8081008,\\
(iv) \quad \mathbb{E}\left[X^8 \right] =&\ 4977\left(\frac{G}{2}\right)^8 \approx 4.5779016,
\end{align*}

with Gau\ss' constant $G = \varpi/\pi$ and the lemniscate constant $\varpi = \Gamma^2(1/4)/(2\sqrt{2\pi})$.
\end{theorem}

\begin{proof}
Let $C_{1,b} = 1/(\sqrt{2}G)$. We get $\mathbb{E}[X^n] = 0$ for odd $n$ due to symmetry. The other formulas directly follow by Ramanujan's \textit{Entries 17 (ii)-(v)} in Chapter 17 of his second notebook, proven by Berndt on pp. 138-139 in \cite{ramanujan_1991}: For $j\in \{2,4,6,8\}$, we write
$$
\mathbb{E}\left[X^{j}\right] = C_{1,b} \sum_{k\in\mathbb{Z}} k^{j} \operatorname{sech}(\pi k) = \frac{\sqrt{2}}{G} \sum_{n=1}^\infty n^{j} \operatorname{sech}(\pi n).
$$
We translate Ramanujan's expressions into the common notation via $x = k^2$, $y = \pi K/K'$, and $z=2 K/\pi$ and use the elliptic lambda function $\lambda^\star(k_1) = 1/\sqrt{2}$ with $K/K'=1$ and $K(k_1)=\varpi/\sqrt{2}$. This yields $x=1/2$, $y=\pi$ and $z=\sqrt{2}G$. We have

$$
\sum_{n=1}^\infty n^{2} \operatorname{sech}(y n) = \frac{1}{8} z^3 x
\ \Rightarrow \
\mathbb{E}\left[X^2\right] = \frac{1}{4} G^2,
$$
$$
\sum_{n=1}^\infty n^{4} \operatorname{sech}(y n) = \frac{1}{2} z^5 \left(\frac{x}{4} + \left(\frac{x}{4} \right)^2 \right)
\ \Rightarrow \
\mathbb{E}\left[X^4\right] = \frac{9}{16} G^4,
$$
$$
\sum_{n=1}^\infty n^{6} \operatorname{sech}(y n) = \frac{1}{2} z^7 \left(\frac{x}{4} + 11\left(\frac{x}{4}\right)^2 + \left(\frac{x}{4} \right)^3 \right)
\ \Rightarrow \
\mathbb{E}\left[X^6\right] = \frac{153}{64} G^6,
$$
$$
\sum_{n=1}^\infty n^{8} \operatorname{sech}(y n) = \frac{1}{2} z^9 \left(\frac{x}{4} + 57\left(\frac{x}{4}\right) + 102 \left(\frac{x}{4} \right)^3 + \left(\frac{x}{4} \right)^4 \right)
\ \Rightarrow \
\mathbb{E}\left[X^8\right] = \frac{4977}{256} G^8,
$$
\vspace{0.7cm}
as claimed.
\end{proof}

Now, we can deduce some exact expressions for series of odd powers of the hyperbolic secant via the Poisson summation.

\begin{theorem}
The following identities hold:

\begin{align*}
(i) \quad \sum_{k\in\mathbb{Z}} \mathrm{sech}^3 (\pi k) =&\ \frac{\sqrt{2}}{2}\left(G^3 + G\right) \approx 1.0012841,\\
(ii) \quad \sum_{k \in \mathbb{Z}} \mathrm{sech}^5(\pi k) =&\ \frac{\sqrt{2}}{24}(9 G^5 + 10 G^3 + 9 G) \approx 1.0000096,\\
(iii) \quad \sum_{k \in \mathbb{Z}} \mathrm{sech}^7(\pi k) =&\ \frac{\sqrt{2}}{720}\left(153 G^7 + 315 G^5 + 259 G^3 + 225 G\right) \approx 1.0000000711,\\
(iv) \quad \sum_{k \in \mathbb{Z}} \mathrm{sech}^9(\pi k) =&\ \frac{\sqrt{2}}{42320}\left(4977 G^9 + 12852 G^7 + 17766 G^5 + 12916 G^3 + 11025 G\right)\\
\approx&\ 1.0000000005,
\end{align*}
with Gau\ss' constant $G = \varpi/\pi$ and the lemniscate constant $\varpi = \Gamma^2(1/4)/(2\sqrt{2\pi})$.
\end{theorem}

\begin{proof}
Let $C_{1,b} = 1/(\sqrt{2}G)$. For $n \in \{2, 4, 6, 8\}$, we apply Poisson's summation formula to the identities $C_{1,b} \sum_{k\in\mathbb{Z}} k^{n} \mathrm{sech}(\pi k)$, i.\,e. the $n$-th moments of a discrete random variable which is distributed according to the first bilateral hyperbolic secant distribution. Define the function $s(x) = x^{n} \cdot \mathrm{sech}(\pi x)$ and the Fourier transform
\begin{align*}
\mathcal{F}[s](\xi) =&\ \int_{-\infty}^\infty s(x) e^{-2\pi i x \cdot \xi} \mathrm{d}x.
\end{align*}
It is easy to see that for every $n\in\mathbb{N}$, we get
$$
\mathcal{F}[x^n f(x)](\xi) = \frac{1}{(-2\pi i)^n} \frac{\partial^n}{\partial\xi^n} [\mathcal{F}[f(x)](\xi)]
$$
for any Schwartz function $f \in \mathcal{S}(\mathbb{R})$ and it is well-known that
$$
\mathcal{F}[\mathrm{sech}(\pi x)](\xi) = \mathrm{sech(\pi \xi)}
$$
holds, i.\,e. this version of the hyperbolic secant is a fix-point of this particular form of the Fourier transform. Now, Poisson's summation formula states
$$\sum_{k\in \mathbb{Z}}s(k) = \sum_{\kappa \in \mathbb{Z}}\mathcal{F}[s](\kappa).$$
\textit{(i)} For $n=2$, we have 
\begin{align*}
\mathcal{F}[s](\xi) = -\frac{1}{4 \pi ^2} \frac{\partial^2}{\partial\xi^2} [\mathrm{sech}(\pi \xi)]
= \frac{1}{4}(2 \cdot \mathrm{sech}^3(\pi \xi) - \mathrm{sech}(\pi \xi)). 
\end{align*}
This results in
$$
\mathbb{E}\left[X^2\right] = \frac{C_{1,b}}{2} \sum_{\kappa \in \mathbb{Z}} \mathrm{sech}^3(\pi \kappa) - \frac{C_{1,b}}{4} \sum_{\kappa \in \mathbb{Z}} \mathrm{sech}(\pi \kappa).
$$
But the last summand can be explicitly calculated thanks to one of Ramanujan's identities, cf. p. 138 of \cite{ramanujan_1991}. It was already used in the proof of Theorem \ref{thm:pmfs}. One has
$$
\sum_{\kappa \in \mathbb{Z}} \mathrm{sech}(\pi \kappa) = 2 \sum_{n=1}^\infty \mathrm{sech}(\pi n) + 1 = 2 \frac{\sqrt{2}G - 1}{2}  + 1 =\sqrt{2}G = \frac{1}{C_{1,b}}.
$$
This finally leads to
$$
\mathrm{Var}(X) = \mathbb{E}\left[X^2\right] = \frac{1}{2\sqrt{2}G} \sum_{\kappa \in \mathbb{Z}} \mathrm{sech}^3(\pi \kappa)  - \frac{1}{4}.
$$
But, result \textit{(i)} from Theorem \ref{thm_Var_1_b} implies 
$$
\mathrm{Var}(X) = \frac{1}{4}G^2,
$$
which holds if and only if
$$
\sum_{k \in \mathbb{Z}} \mathrm{sech}^3(\pi k) = \frac{G^3 + G}{\sqrt{2}}.
$$
\textit{(ii)} For $n=4$, we have 
\begin{align*}
\mathcal{F}[s](\xi) =&\ \frac{1}{16 \pi ^4} \frac{\partial^4}{\partial\xi^4} \left[ \mathrm{sech}(\pi \xi)\right] = \frac{3}{2} \operatorname{sech}^5(\pi \xi) - \frac{5}{4}\operatorname{sech}^3 (\pi \xi) + \frac{1}{16} \operatorname{sech}(\pi \xi).
\end{align*}
The Poisson summation results in
$$
\mathbb{E}\left[X^4\right] = \frac{3}{2} C_{1,b} \sum_{\kappa \in \mathbb{Z}} \mathrm{sech}^5(\pi \kappa) - \frac{5}{4} C_{1,b} \sum_{\kappa \in \mathbb{Z}} \mathrm{sech}^3(\pi \kappa) + \frac{1}{16} C_{1,b} \sum_{\kappa \in \mathbb{Z}} \mathrm{sech}(\pi \kappa).
$$
But the last two series can be explicitly calculated by part \textit{(i)}, viz.
$$
\sum_{\kappa \in \mathbb{Z}} \mathrm{sech}^3(\pi \kappa) = \frac{G^3 + G}{\sqrt{2}}, \quad \sum_{\kappa \in \mathbb{Z}} \mathrm{sech}(\pi \kappa) = \frac{1}{C_{1,b}} = \sqrt{2}G.
$$
This yields
$$
\mathbb{E}\left[X^4\right] = \frac{3}{2\sqrt{2}G} \sum_{\kappa \in \mathbb{Z}} \mathrm{sech}^5(\pi \kappa) - \frac{5}{4\sqrt{2}G} \frac{G^3 + G}{\sqrt{2}} + \frac{1}{16}.
$$
But, result \textit{(ii)} from Theorem \ref{thm_Var_1_b} tells us
$$
\mathbb{E}\left[X^4\right] = \frac{9}{16}G^4,
$$
which holds if and only if
$$
\sum_{k \in \mathbb{Z}} \mathrm{sech}^5(\pi k) = \frac{\sqrt{2}}{24}(9 G^5 + 10G^3+ 9G).
$$

\textit{(iii)} For $n=6$, we have
\begin{align*}
\mathcal{F}[s](\xi) =&\ -\frac{1}{64 \pi ^6} \frac{\partial^6}{\partial\xi^6} \left[ \mathrm{sech}(\pi \xi)\right]\\
=&\ \frac{45}{4} \operatorname{sech}^7(\pi \xi) - \frac{105}{8}\operatorname{sech}^5 (\pi \xi) + \frac{91}{32} \operatorname{sech}^3(\pi \xi) - \frac{1}{64} \operatorname{sech}(\pi \xi).
\end{align*}
The Poisson summation results in
\begin{align*}
\mathbb{E}\left[X^6\right] =\ &\frac{45}{4} C_{1,b} \sum_{\kappa \in \mathbb{Z}} \operatorname{sech}^7(\pi \kappa) - \frac{105}{8} C_{1,b} \sum_{\kappa \in \mathbb{Z}} \operatorname{sech}^5 (\pi \kappa)\\
&+ \frac{91}{32} C_{1,b} \sum_{\kappa \in \mathbb{Z}} \operatorname{sech}^3(\pi \kappa) - \frac{1}{64} C_{1,b} \sum_{\kappa \in \mathbb{Z}} \operatorname{sech}(\pi \kappa).
\end{align*}
But the last four series can be explicitly calculated by the parts \textit{(i)} and \textit{(ii)}, viz.
\begin{align*}
& \sum_{k \in \mathbb{Z}} \mathrm{sech}^5(\pi k) = \frac{\sqrt{2}}{24}\left(9 G^5 + 10 G^3 + 9 G\right), \\ & \sum_{\kappa \in \mathbb{Z}} \mathrm{sech}^3(\pi \kappa) = \frac{\sqrt{2}}{2} \left(G^3+G\right), \\ & \sum_{\kappa \in \mathbb{Z}} \mathrm{sech}(\pi \kappa) = \sqrt{2}G.
\end{align*}
This yields
$$
\mathbb{E}\left[X^6\right] = \frac{45}{4\sqrt{2}G} \sum_{\kappa \in \mathbb{Z}} \mathrm{sech}^7(\pi \kappa) - \frac{105}{192} \left(9 G^4 + 10 G^2 + 9\right) + \frac{91}{64} \left(G^2 +1 \right) - \frac{1}{64}.
$$
But, result \textit{(ii)} from Theorem \ref{thm_Var_1_b} tells us
$$
\mathbb{E}\left[X^6\right] = \frac{153}{64}G^6,
$$
which holds if and only if
$$
\sum_{k \in \mathbb{Z}} \mathrm{sech}^7(\pi k) = \frac{\sqrt{2}}{720}\left(153 G^7 + 315 G^5 + 259 G^3 + 225 G\right).
$$

\textit{(iv)} For $n=8$, we have
\begin{align*}
\mathcal{F}[s](\xi) =&\ -\frac{1}{256 \pi^8} \frac{\partial^8}{\partial\xi^8} \left[ \mathrm{sech}(\pi \xi)\right]\\
=&\ \frac{315}{2} \operatorname{sech}^9(\pi \xi) - \frac{945}{4} \operatorname{sech}^7(\pi \xi) + \frac{1449}{16}\operatorname{sech}^5 (\pi \xi) - \frac{205}{32} \operatorname{sech}^3(\pi \xi) + \frac{1}{256} \operatorname{sech}(\pi \xi).
\end{align*}
The Poisson summation formula results in
\begin{align*}
\mathbb{E}\left[X^8\right] = \sum_{\kappa \in \mathbb{Z}} \Bigg[ & \frac{315}{2\sqrt{2}G}\operatorname{sech}^9(\pi \kappa) - \frac{945}{4\sqrt{2}G} \operatorname{sech}^7 (\pi \kappa) + \frac{1449}{16\sqrt{2}G} \operatorname{sech}^5(\pi \kappa)\\ 
&- \frac{205}{32\sqrt{2}G} \operatorname{sech}^3(\pi \kappa) + \frac{1}{256 \sqrt{2}G} \operatorname{sech}(\pi \kappa) \Bigg].
\end{align*}
But the last four series can be explicitly calculated by our parts \textit{(i)}, \textit{(ii)} and \textit{(iii)}, viz.
\begin{align*}
& \sum_{k \in \mathbb{Z}} \mathrm{sech}^7(\pi k) = \frac{\sqrt{2}}{720}\left(153 G^7 + 315 G^5 + 259 G^3 + 225 G\right)\\
& \sum_{k \in \mathbb{Z}} \mathrm{sech}^5(\pi k) = \frac{\sqrt{2}}{24}\left(9 G^5 + 10 G^3 + 9 G\right), \\ & \sum_{\kappa \in \mathbb{Z}} \mathrm{sech}^3(\pi \kappa) = \frac{\sqrt{2}}{2} \left(G^3+G\right), \\ & \sum_{\kappa \in \mathbb{Z}} \mathrm{sech}(\pi \kappa) = \sqrt{2}G.
\end{align*}

But, result \textit{(ii)} from Theorem \ref{thm_Var_1_b} tells us
$$
\mathbb{E}\left[X^8\right] = \frac{4977}{256}G^8,
$$
which holds if and only if
$$
\sum_{k \in \mathbb{Z}} \mathrm{sech}^9(\pi k) = \frac{\sqrt{2}}{42320}\left(4977 G^9 + 12852 G^7 + 17766 G^5 + 12916 G^3 + 11025 G\right),
$$
as claimed.
\end{proof}

\vspace{1cm}

\section{The Second Discrete Bilateral Hyperbolic Secant Distribution}

Ramanujan's {Entries 16 (x) - (xiii) in Chapter 17 of his second notebook, proven by Berndt on pp. 137-138 in \cite{ramanujan_1991}, can be statistically interpreted as follows: If you randomly generate odd integers with even multiplicity by using the second bilateral hyperbolic secant distribution, you can calculate the expected outcome up to the eighth order:

\begin{theorem} 
Let $Y\sim\mathsf{Sech}_{2,b}$ be a discrete random variable which is distributed according to the second bilateral hyperbolic secant distribution. We have
\begin{align*}    
(i) \quad \mathbb{E}\left[(2Y+1)^2 \right] =&\ 2G^2 \approx 1.3932039,\\
(ii) \quad \mathbb{E}\left[(2Y+1)^4 \right] =&\ 12 G^4 \approx 5.8230516,\\
(iii) \quad \mathbb{E}\left[(2Y+1)^6 \right] =&\ 216 G^6 \approx 73.0142850,\\
(iv) \quad \mathbb{E}\left[(2Y+1)^8 \right] =&\ 7056 G^8 \approx 1661.4885492,
\end{align*}

with Gau\ss' constant $G = \varpi/\pi$ and the lemniscate constant $\varpi = \Gamma^2(1/4)/(2\sqrt{2\pi})$.
\end{theorem}

\begin{proof}
Let $C_{2,b} = 1/G$. We get $\mathbb{E}[Y^j] = 0$ for odd $j \in \mathbb{N}$ due to symmetry. The other formulas directly follow by Ramanujan's \textit{Entries 16 (x)-(xiii)} in Chapter 17 of his second notebook, proven by Berndt on pp. 137-138 in \cite{ramanujan_1991}: For $j\in \{2,4,6,8\}$, we write
$$
\mathbb{E}\left[Y^{j}\right] = C_{2,b} \sum_{k\in\mathbb{Z}} (2k+1)^{j} \operatorname{sech}\left(\frac{\pi}{2} (2k+1)\right) = \frac{2}{G} \sum_{n=0}^\infty (2n+1)^{j} \operatorname{sech}\left(\frac{\pi}{2} (2n+1)\right).
$$
We translate Ramanujan's expressions into the common notation via $x = k^2$, $y = \pi K/K'$, and $z=2 K/\pi$ and use the elliptic lambda function $\lambda^\star(k_1) = 1/\sqrt{2}$ with $K/K'=1$ and $K(k_1)=\varpi/\sqrt{2}$. This yields $x=1/2$, $y=\pi$ and $z=\sqrt{2}G$. We have

\begin{align*}
\textit{(i)} \quad
\sum_{n=0}^\infty (2n+1)^{2} \operatorname{sech}\left(\frac{1}{2} (2n+1)y\right) &= \frac{1}{2} z^3 \sqrt{x}\\
& \Rightarrow \mathbb{E}\left[(2Y+1)^2 \right]= 2 G^2,\\
\textit{(ii)} \quad
\sum_{n=0}^\infty (2n+1)^{4} \operatorname{sech}\left(\frac{1}{2} (2n+1)y\right) &= \frac{1}{2} z^5 (1+4x) \sqrt{x}\\
& \Rightarrow \mathbb{E}\left[(2Y+1)^4 \right] = 12 G^4,\\
\textit{(iii)} \quad
\sum_{n=0}^\infty (2n+1)^{6} \operatorname{sech}\left(\frac{1}{2} (2n+1)y\right) &= \frac{1}{2} z^7 (1+11(4x) + (4x)^2) \sqrt{x}\\
& \Rightarrow \mathbb{E}\left[(2Y+1)^6 \right] = 216 G^6,\\
\textit{(iv)} \quad
\sum_{n=0}^\infty (2n+1)^{8} \operatorname{sech}\left(\frac{1}{2} (2n+1)y\right) &= \frac{1}{2} z^9 (1 + 102(4x) + 57(4x)^2 + (4x)^3) \sqrt{x}\\
& \Rightarrow \mathbb{E}\left[(2Y+1)^8 \right] = 7056 G^8,
\end{align*}
as claimed.
\end{proof}

\section{Application to Contour Integrals}

We apply our result to complex analysis and calculate some contour integral on the critical line in terms of the variance of our first discrete bilateral hyperbolic secant distribution. We integrate the product of the analytic continuation of Dirichlet's beta function $\beta$ and the meromorphic continuations of Euler's Gamma function $\Gamma$ and Riemann's zeta function $\zeta$. Further, we use Euler's numbers which are defined via
$$
\operatorname{sech}(t) = \sum_{n=0}^\infty \frac{E_n}{n!} \cdot t^n,$$
vanishing for odd $n$.

\begin{theorem}
Let $\Xi(s):= \pi^{-(s+2)} \cdot \beta(s+2) \cdot \Gamma(s+2) \cdot \zeta(s)$. On the critical line, one has
\begin{align*}
\int_{1/2-i\cdot \infty}^{1/2 + i \cdot \infty} \Xi(s)\, \mathrm{d}s = i \left( \frac{\varpi}{\sqrt{2}} \cdot \mathbb{E}\left[X^2 \right] - \frac{\pi}{8} \right) = i \left( \frac{\varpi}{\sqrt{2}} \left(\frac{G}{2}\right)^2 - \frac{\pi}{8} \right) \approx&\ -i \cdot 0.0698111,
\end{align*}
with Gau\ss' constant $G = \varpi/\pi$ and the lemniscate constant $\varpi = \Gamma^2(1/4)/(2\sqrt{2\pi})$.
\end{theorem}

\begin{proof}
Let $X\sim\mathsf{Sech}_{1,b}$ be a discrete random variable which is distributed according to the first bilateral hyperbolic secant distribution. We expand the second moment of $X$. Let $C_{1,b} = 1/(\sqrt{2}G)$. By definition, one has
$$
\mathbb{E}\left[X^2\right] = C_{1,b} \sum_{k \in \mathbb{Z}} k^2\cdot \operatorname{sech}\left(\pi k\right) = 2C_{1,b} \sum_{n=1}^\infty n^2 \cdot \operatorname{sech}(\pi n).
$$
For $x > 0$, one has $|\exp(-2x)| < 1$. Hence, the geometric series yields the identities
$$
\operatorname{sech}(x) = \frac{2e^{-x}}{1+e^{-2x}} = 2e^{-x} \sum_{k=0}^\infty (-1)^k e^{-2kx} = 2 \sum_{k=0}^\infty (-1)^k e^{-(2k+1)x}.
$$
Now, the gamma function motivates the introduction of the Mellin transform $\mathcal{M}$ on a strip $\mathbb{S} := \{0 < \Re(\sigma) < b\}$. For every $f \in L^1_{\mathrm{loc}}((\mathbb{S},\infty), \mathbb{C})$, it is defined as
$$
\mathcal{M}\left\{f\right\}(s) = \int_0^\infty x^{s-1} f(x) \, \mathrm{d}x, \quad \mathrm{with} \quad \mathcal{M}\left\{\left[x \mapsto e^{-x}\right]\right\} = \int_0^\infty x^{s-1}e^{-x} \, \mathrm{d}x = \Gamma(s).
$$
The gamma function is analytic on $\mathbb{S}$ and Stirling's formula
\begin{equation}\label{Stirlings_formula}
|\Gamma(\sigma + i \cdot t)| \in O\left(\sqrt{2\pi} \cdot |t|^{\sigma-1/2} e^{-|t| \pi/2}\right) (|t| \rightarrow \infty)
\end{equation}
holds uniformly for any $\sigma > 0$. Hence, the Mellin inversion theorem yields the existence of the Cahen-Mellin integral
$$
e^{-x} = \mathcal{M}^{-1}\left\{\Gamma\right\}(x) = \frac{1}{2\pi i} \int_{c-i \cdot\infty}^{c+i \cdot\infty} \Gamma(s) \cdot x^{-s} \, \mathrm{d}s.
$$
For any fixed $c>0$ we conclude
$$
\mathbb{E}\left[X^2\right] = 4C_{1,b} \sum_{n=1}^\infty n^2 \sum_{k=0}^\infty (-1)^k \frac{1}{2\pi i} \int_{c-i \cdot\infty}^{c+i \cdot\infty} \frac{\Gamma(s)}{(2k+1)^{s}} \cdot (\pi n)^{-s}\, \mathrm{d}s.
$$
Next, we want to interchange the limits. On the line $\Re(s)=c>3$, we get
$$
\left|(-1)^k \cdot \frac{\Gamma(s)\cdot n^{2-s}}{(\pi (2k+1))^{s}}\right| = \frac{\Gamma(c)\cdot n^{2-c}}{\left(\pi (2k+1)\right)^{c}}
$$
and by recalling Stirling's formula (\ref{Stirlings_formula}) we infer
$$
\int_{c-i \cdot\infty}^{c+i \cdot\infty} |\Gamma(s)|\, \mathrm{d}s < \infty, \quad \sum_{k=0}^\infty (2k+1)^{-c} < \infty, \quad \sum_{n=1}^\infty n^{2-c} < \infty.
$$
Thus, Tonelli's theorem results in
\begin{align*}
\mathbb{E}\left[X^2\right] =&\ \frac{4C_{1,b}}{2\pi i} \int_{c-i \cdot\infty}^{c+i \cdot\infty} \frac{\Gamma(s)}{\pi^s} \sum_{k=0}^\infty \frac{(-1)^k}{(2k+1)^{s}} \sum_{n=1}^\infty \frac{1}{n^{s-2}}\, \mathrm{d}s \\
=&\ \frac{4C_{1,b}}{2\pi i} \int_{c-i \cdot\infty}^{c+i \cdot\infty} \underbrace{\pi^{-s} \cdot \beta(s) \cdot \Gamma(s) \cdot \zeta(s - 2)}_{=:\, \Phi(s)} \, \mathrm{d}s.
\end{align*}

We want to solve the integral by the method of residues by shifting the vertical line of integration to the left. Therefore we consider the analytic continuation of $\beta$ and the meromorphic continuations of $\Gamma$, and $\zeta$. In the representation $\Phi(s)$ of the integrand, the simple poles of $\Gamma$ and the trivial zeros of $\zeta(\cdot  - 2)$ cancel each other out for even non-positive integers $s\in \{2-2n\}_{n\in \mathbb{N}}$. In addition, the simple poles of $\Gamma$ and the simple zeros of $\beta$ cancel each other out for odd negative integers $s \in \{1-2n\}_{n \in \mathbb{N}}$. Hence, all $\Gamma$-poles in $s\in\{-n\}_{n\in \mathbb{N}_0}$ are canceled out and only the simple pole at $s = \varrho = 3$ from $\zeta(\cdot - 2)$ remains with residue
$$
\operatorname{Res}_{\Phi}(\varrho = 3) = \pi^{-3} \cdot \beta(3) \cdot \Gamma(3) \cdot 1 = \frac{2!}{\pi^{3}} \beta(3) = \frac{- E_2}{4^2} = \frac{1}{4^2},
$$
since the identity
$$
\beta(r+1)=\frac{(-1)^{\frac{r}{2}} E_{r}}{2\cdot r!}\left(\frac\pi2\right)^{r+1}
$$
with Euler's numbers $E_r$ holds for even $r \in \mathbb{N}$ and one has $E_2 = -1$.\\

We introduce a new contour for the integral of $\Phi$. Let $c=3+\varepsilon$ for some $\varepsilon>0$ and let $R= R_{5/2,c}(\pm T)$ be the rectangle with vertices $c \pm i\cdot T$ and $5/2 \pm i \cdot T$. Then,by the residue theorem, one has
$$
\frac{4 C_{1,b}}{2\pi i}\int_{R} \Phi(s) \, \mathrm{d}s = 4 C_{1,b} \sum_{\varrho \in R} \operatorname{Res}_\Phi(\varrho) = \frac{C_{1,b}}{4},
$$
while
\begin{align*}
&\ \frac{4 C_{1,b}}{2\pi i}\int_{c-i\cdot T}^{c+i\cdot T} \Phi(s) \mathrm{d}s\\
=&\ \frac{4 C_{1,b}}{2\pi i} \left[\underbrace{\int_{5/2 + i\cdot T}^{c + i\cdot T} \Phi(s) \mathrm{d}s}_{=: I_1} + \underbrace{\int_{5/2-i\cdot T}^{5/2+i\cdot T} \Phi(s) \mathrm{d}s}_{=: I_2} + \underbrace{\int_{5/2 - i \cdot T}^{c- i \cdot T} \Phi(s) \mathrm{d}s}_{=: I_3}\right] + \frac{C_{1,b}}{4}.
\end{align*}

We show that the integrals over the horizontal edges vanish: For $\sigma > 1$, the beta function converges absolutely. Hence, one has
$$
\lvert I_1\rvert, \lvert I_3\rvert \le \int_{5/2 \pm i\cdot T}^{c \pm i\cdot T} \lvert \pi^{-s} \beta(s) \Gamma(s) \zeta(s-2) \rvert \, \mathrm{d}s \lesssim \int_{5/2 \pm i\cdot T}^{c \pm i\cdot T} \lvert \Gamma(s) \rvert \cdot \zeta(\sigma-2) \, \mathrm{d}s.
$$
The zeta term grows at most exponentially and Stirling's formula (\ref{Stirlings_formula}) results in 
$$\lvert \Phi(\sigma + i \cdot T)\rvert \in O\left(T^{\sigma/2} e^{-T\pi/2}\right)(T \rightarrow \infty).$$
Therefore, $\sup_{\sigma \in [5/2, 3+\varepsilon]}\lvert \Phi(\sigma + i \cdot T) \rvert \rightarrow 0$ as $T \rightarrow \infty$ and $I_1$ and $I_3$ uniformly vanish in the limit. Now, we consider the shifted function 
$$\Xi(s):= \Phi(s+2) = \pi^{-(s+2)} \cdot \beta(s+2) \cdot \Gamma(s+2) \cdot \zeta(s) $$ 
on the line $\Re(s)=1/2.$ First, we show that $\lim_{T \rightarrow \infty} I_2$ exists. The beta function converges absolutely for $\sigma > 1$. Further, Bourgain showed in \cite{bourgain_2016}, that
$$
\zeta(1/2 + i \cdot t) \in O\left(\lvert t\rvert^{\frac{13}{84}+\delta}\right) (\lvert t\rvert \rightarrow \infty)
$$
for any $\delta > 0$. Together with Stirling's formula, those facts result in
$$\Xi\left(\frac{1}{2} \pm i \cdot T\right) \in O\left(T^{2 + \frac{13}{84} + \delta} e^{-T \pi/2}\right) (T\rightarrow\infty).$$
Hence, due to the exponential decay, the expression 
$$\mathbb{E}\left[X^2 \right] = \frac{2 C_{1,b}}{\pi i} \int_{1/2-i\cdot \infty}^{1/2 + i \cdot \infty} \Xi(s) \mathrm{d}s + \frac{C_{1,b}}{4}$$
is well-defined. Now, Theorem \ref{thm_Var_1_b} says $\mathbb{E}\left[X^2\right] = (G/2)^2$. Hence, in terms of the lemniscate constant $\varpi$ and the lemniscate modulus $k=1/\sqrt{2}$, we get the formula
\begin{align*}
\int_{1/2-i\cdot \infty}^{1/2 + i \cdot \infty} \Xi(s) \mathrm{d}s = i \left( \frac{\varpi}{\sqrt{2}} \cdot \mathbb{E}\left[X^2 \right] - \frac{\pi}{8} \right)= i \left( \frac{\varpi}{\sqrt{2}} \left(\frac{G}{2}\right)^2 - \frac{\pi}{8} \right),
\end{align*}
as claimed.
\end{proof}

\end{document}